\documentclass[11pt,a4paper]{amsart}

\usepackage{amssymb,amsmath,amsthm}
\usepackage{commath}
\usepackage{mathtools}
\usepackage{mathrsfs}
\usepackage{accents}
\usepackage{pgfplots}
\usepackage{accessibility}

\usepackage[utf8]{inputenc}

\usepackage{caption}
\usepackage{subcaption}

\usepackage[hyphens]{url}
\usepackage[colorlinks]{hyperref}
\hypersetup{
    allcolors = black
}
\usepackage[hyphenbreaks]{breakurl}
\usepackage[lined, linesnumbered, ruled]{algorithm2e}

\usepackage{xcolor}

\pgfplotsset{compat=1.18}

\usepackage{enumitem}
\makeatletter
\def\namedlabel#1#2{\begingroup
    #2
    \def\@currentlabel{#2}
    \phantomsection\label{#1}\endgroup
}

\newtheorem{theorem}{Theorem}[]
\newtheorem{corollary}{Corollary}

\newtheorem{conj}{Conjecture}

\newtheorem{defi}[]{Definition}
\newtheorem{nota}{Notation}
\newtheorem{remark}{Remark}
\newtheorem{example}{Example}
\newtheorem{construction}{Construction}

\renewcommand{\geq}{\geqslant}
\renewcommand{\leq}{\leqslant}

\newcommand{\R}{\mathbb{R}}
\newcommand{\Q}{\mathbb{Q}}
\newcommand{\C}{\mathbb{C}}
\newcommand{\Z}{\mathbb{Z}}

\newcommand{\eps}{\varepsilon}

\newboolean{appendix_included}
\setboolean{appendix_included}{true}

\newcommand{\appendixmacro}[1]{%
  \ifthenelse{\boolean{appendix_included}}%
    {#1}%
    {\cite[#1]{web}}%
}

\makeatletter
\@namedef{subjclassname@2020}{
  \textup{2020} Mathematics Subject Classification}
\makeatother

\author[M. Matolcsi]{Máté Matolcsi}\thanks{Corresponding author: Máté Matolcsi, e-mail address: \texttt{matomate@renyi.hu}\\[3 pt]}
\author[I. Z. Ruzsa]{Imre Z. Ruzsa}
\author[D. Varga]{Dániel Varga}
\author[P. Zsámboki]{Pál Zsámboki}

\title{The fractional chromatic number of the plane is at least 4} 

\date{\today}

\subjclass[2020]{52C10, 52C15, 05C72}

\keywords{Chromatic number of the plane, distance-avoiding sets, fractional chromatic number, linear programming}

\begin{document}

\begin{abstract}

We prove that the fractional chromatic number $\chi_f(\R^2)$ of the unit distance graph of the Euclidean plane is greater than or equal to $4$. Interestingly, however, we cannot present a finite subgraph $G$ of the plane such that $\chi_f(G)\ge 4$. Instead, we utilize the concept of the geometric fractional chromatic number $\chi_{gf}(G)$, which was introduced recently in connection with density bounds for 1-avoiding sets. 

First, as $G$ ranges over finite subgraphs of the plane, we establish that the supremum of $\chi_f(G)$ is the same as that of $\chi_{gf}(G)$. The proof exploits the amenability of the group of Euclidean transformations in dimension 2 and, as such, we do not know whether the analogous statement holds in higher dimensions. We then present a specific planar unit distance graph $G$ on 27 vertices such that $\chi_{gf}(G)=4$, and conclude $\chi_f(\R^2)\ge 4$ as a corollary. 

As another main result we show that the finitary fractional chromatic number and the Hall ratio of the plane are equal. As a consequence, we conclude that there exist finite unit distance graphs with independence ratio $\frac{1}{4}+\varepsilon$, while we conjecture that the value $\frac{1}{4}$ cannot be reached. 
\end{abstract}

\maketitle

\section{Introduction}

The famous Hadwiger--Nelson problem seeks to determine the chromatic number of the unit distance graph of the Euclidean plane, i.e. the minimum number of colours needed to colour the points of $\R^2$ such that any pair of points at distance 1 have different colours. A major breakthrough was achieved by de Grey \cite{deG18}, showing that at least 5 colours are needed. Some related questions and ideas were subsequently  investigated further by the PolyMath16 project \cite{PM22}. 

\medskip

One such related problem is to determine the {\it fractional chromatic number} $\chi_f(\R^2)$ of the plane (we will recall all necessary definitions in Section \ref{sec2}). Despite considerable efforts, no finite unit distance graph $G$ has yet been found such that $4\le \chi_f(G)$. After the initial non-trivial results \cite{mah, sche, cra} improving on the fractional chromatic number 3.5 of the Moser spindle, the currently best published lower bound of $3.8991\le \chi_f(\R^2)$ belongs to Bellitto, Pêcher, and Sédillot \cite{BePS21}. Some further (unpublished) improvements reaching the value $3.9898\le \chi_f(\R^2)$ were made by Jaan Parts \cite{Pa22}. These latter estimates are testified by finite graphs having hundreds of (or even over a thousand)  vertices. The best upper bound on $\chi_f(\R^2)$ is given in \cite{hod, sche} as the reciprocal of the maximal known density $\delta_\mathrm{Croft}$ of a  measurable planar 1-avoiding set constructed by Croft \cite{Cr67}: 
\begin{equation}\label{fcnu}
\chi_f(\R^2)\le 1/\delta_\mathrm{Croft}< 4.36.     
\end{equation}

\medskip

The main result of this paper is the new bound $\chi_f(\R^2)\ge 4$, established in Corollary \ref{cor}. Interestingly, however, we are tempted to conjecture that this bound cannot be testified by any finite graph, i.e. we have $\chi_f(G)< 4$ for all finite unit distance graphs $G$ in the plane. 

\medskip

The key ingredient of the proof is the notion of the {\it geometric fractional chromatic number} $\chi_{gf}(G)$ of unit distance graphs, introduced recently in \cite{ambrus2023density}. The  idea is that one can construct fairly small graphs $G$ such that $\chi_{gf}(G)=4$, and then it is possible to employ a "blow-up" procedure to find a much larger graph $G'$ with $\chi_f(G')\ge 4-\eps$ for any $\eps>0$. 

\medskip

In another main result of the paper, Theorem \ref{alpha}, we prove that the finitary fractional chromatic number and the Hall ratio of the unit distance graph of the plane are equal. For general graphs, the connection between the fractional chromatic number and the Hall ratio has been studied extensively. Recent breakthroughs \cite{blu, dvo} show that the fractional chromatic number, in general, cannot be upper bounded by a constant multiple of the Hall ratio. These results disprove eariler conjectures in \cite{john, har}, and also have implications in the theory of bounded expansion \cite{nes}. In fact, for graphs on $n$ vertices, $\chi_f(G)/\rho(G)$ can be as large as $(\log n)^{1-o(1)}$, as proven in \cite{steiner} based on a modification of a construction in \cite{jan}.

\medskip

This paper is organized as follows. In Section \ref{sec2} we recall all the necessary notions, and introduce notations. In Section \ref{sec3} we first prove that finitary fractional chromatic number and the finitary geometric fractional chromatic number of the unit distance graph of the plane are equal, i.e. $\sup\{\chi_{gf}(G): \ G\subseteq \R^2, \ G \ \text{is finite}\}=\sup\{\chi_f(G): \ G\subseteq \R^2, \  G \ \text{is finite}\}$. Next we prove that the finitary fractional chromatic number equals the Hall ratio, i.e. 
 $\sup\{\chi_{f}(G): \ G\subseteq \R^2, \ G \ \text{is finite}\}=\sup\{|G|/\alpha(G): \  G\subseteq \R^2, \  G \ \text{is finite}  \}$. Finally, in Section \ref{sec4} we present a graph $G$ with 27 vertices such that $\chi_{gf}(G)=4$. As a direct corollary, we obtain $\chi_f(\R^2)\ge 4$.

\section{Preliminaries}\label{sec2}

Throughout the paper the term 'graph' always refers to a unit distance graph $G$ in the plane, i.e.~a graph given by a vertex set $X\subseteq \R^2$ that has an edge between two vertices $x, y\in X$ if and only if their Euclidean distance is 1. The vertex set $X$ is usually finite, but in certain cases it may be infinite.  Sometimes we may abuse notation, and use the letters $X$ and $G$ interchangeably to refer to the same graph. 

\medskip

Given a finite unit distance graph $G$, denote by $\mathcal {I}(G)$ the set of all  independent sets of $G$, and let $\mathcal {I}(G,x)$ be the set of all those independent sets which include the vertex $x \in G$. 

\medskip 

A {\it fractional colouring} of $G$ is a function $\gamma: \mathcal {I}(G)\to \R$ such that 

\begin{equation}\label{noneg}
\gamma(S)\ge 0 \textrm{\ for\ all\ } S\in \mathcal {I}(G),
\end{equation}
and
\begin{equation} \label{agix}
\sum_{S\in\mathcal{I}(G,x)} \gamma(S) \ge 1
\end{equation}
for every vertex $x \in G$.

\medskip

To avoid any possible confusion later, we extend the domain of $\gamma$ to include all subsets of $G$, and define  $\gamma(S)=0$ for all non-independent subsets $S$. 

\medskip

The quantity $\sum_{S\in \mathcal{I}(G)} \gamma(S)$ is called the {\it weight} of $\gamma$.

\medskip

The set of all fractional colourings of $G$ will be denoted by $fc(G)$.

\medskip

In order to enable easier subsequent comparison with geometric fractional colourings, we call a fractional colouring $\gamma$ of $G$ {\it regular}, if

\addtocounter{equation}{-1} 
\renewcommand{\theequation}{\arabic{equation}'} 

\begin{equation} \label{agix_prime}
\sum_{S\in\mathcal{I}(G,x)} \gamma(S) = 1
\end{equation}
for every vertex $x \in G$.

\renewcommand{\theequation}{\arabic{equation}} 

\medskip

The set of all regular fractional colourings of $G$ will be denoted by $rfc(G)$. 

\medskip

The {\em fractional chromatic number} of $G$ is defined as
\begin{equation} \label{agi}
\chi_f(G)=\min_{\gamma\in fc(G)} \sum _{S\in \mathcal{I}(G)} \gamma(S),
\end{equation}
that is, the minimum  weight of a fractional colouring of $G$. It is fairly easy to see, as shown in \cite[Lemma 2]{ambrus2023density}, that the minimum value does not change if $\gamma$ ranges over all regular fractional colourings only. That is 
\begin{equation}\label{tfc}
 \chi_f(G)=\min_{\gamma\in fc(G)} \sum _{S\in \mathcal{I}(G)} \gamma(S)=\min_{\gamma\in rfc(G)} \sum _{S\in \mathcal{I}(G)} \gamma(S)   
\end{equation}

\medskip

A {\it geometric fractional colouring} of $G$ is a regular fractional colouring  (i.e.  satisfying \eqref{noneg}, \eqref{agix_prime}), with the additional constraint 
\begin{equation}\label{congr}
\sum_{Y\subseteq S}\gamma(S) = \sum_{Y'\subseteq S'}\gamma(S')
\end{equation}
for all geometrically congruent subsets $Y, Y'$ of $G$.

\medskip

The set of all geometric fractional colourings of $G$ will be denoted by $gfc(G)$.

\medskip

The {\em geometric fractional chromatic number} of $G$, denoted by $\chi_{gf}(G)$, as introduced in \cite[Definition 2]{ambrus2023density}, is defined as 
\begin{equation*}
\chi_{gf}(G)=\min_{\gamma\in gfc(G)} \sum _{S\in \mathcal{I}(G)} \gamma(S),
\end{equation*}
that is, the minimum weight of a geometric fractional colouring of $G$. 

\medskip

Due to the inclusion $gfc(G)\subseteq fc(G)$, we clearly have $\chi_{gf}(G) \geq \chi_f(G)$ for any finite graph $G$. Furthermore, by a standard continuity and compactness argument, the following approximation property holds: for any finite $G$ and any $\eps>0$ there exists a $\delta>0$ such that if a regular fractional colouring $\gamma$ satisfies 
\begin{equation}\label{delta}
\left|\sum_{Y\subseteq S}\gamma(S) - \sum_{Y'\subseteq S'}\gamma(S')\right|<\delta
\end{equation}
for all congruent subsets $Y, Y'\subseteq G$ then 
\begin{equation}\label{epsd}
\sum_S \gamma(S)> \chi_{gf}(G)-\eps.
\end{equation}

\medskip

Given a fractional colouring $\gamma$ of $G$, we define its {\it aggregate functional} $\overline{\gamma}$ as follows:  for any subset $S\subseteq G$ let 
\begin{equation}\label{aggr}
\overline{\gamma}(S)=\sum_{S'\supset S} \gamma(S')
\end{equation}

Note, in particular, that $\overline{\gamma}(\emptyset)=\sum _{S\in \mathcal{I}(G)} \gamma(S)$ is the weight of $\gamma$.

\medskip

It is easy to see that we can recover $\gamma$ from its aggregate functional $\overline{\gamma}$ via an inclusion-exclusion formula: 
$$\gamma(S)=\sum_{S\subseteq S'} \overline{\gamma}(S') (-1)^{|S'\setminus S|}$$

\medskip

For any subgraph $H\subseteq G$, any fractional colouring $\gamma$ on $G$ induces a fractional colouring $\gamma_H$ on $H$ as follows: for any $S\subseteq H$ let
$$ \gamma_H(S)=\sum_{Y\subseteq G\setminus H} \gamma(S\cup Y)$$
It is easy to see that $\gamma_H$ is indeed a fractional colouring of $H$. Also, there is an evident relationship between the corresponding aggregate functionals: 
 for any $S\subseteq H$ we have 
 \begin{equation}\label{le}
 \overline{\gamma_H}(S)=\overline{\gamma}(S).
\end{equation}

\medskip

In particular, the weights of $\gamma$ and $\gamma_H$ are the same: 

\begin{equation}\label{weight}
\sum_{S\subseteq G} \gamma(S)=\overline{\gamma}(\emptyset)=\overline{\gamma_H}(\emptyset)=\sum_{S\subseteq H} \overline{\gamma_H}(S).    
\end{equation}
This immediately implies that both the fractional chromatic number and the geometric fractional chromatic number are monotonically increasing quantities with respect to inclusion: if $H\subseteq G$ then $\chi_f(H)\le \chi_f(G)$ and $\chi_{gf}(H)\le \chi_{gf}(G)$. 

\medskip

The {\it independence number} $\alpha(G)$ of a finite graph $G$ is the maximal number of independent vertices in $G$. The {\it independence ratio} of $G$ is given by $\alpha(G)/|G|$, and will be denoted by $\alpha_1(G)$. It is well known (see e.g. \cite[Proposition 3.1.1]{sche}) that 
\begin{equation} \label{fcna}
\chi_f(G) \alpha_1(G)\ge 1
\end{equation}
for every finite graph $G$. The {\it Hall ratio} $\rho(G)$ of $G$ is defined as $\rho(G)=\max_{H\subseteq G} \frac{1}{\alpha_1(H)}$. Equation \eqref{fcna} and the monotonically increasing property of the fractional chromatic number imply the well-known inequality  $\rho(G)\le \chi_f(G)$. 

\medskip

For the infinite unit distance graph of the Euclidean plane $\R^2$, the situation is more delicate. On the one hand, a well-known result of de Bruijn and Erd\H os \cite{edb} says (assuming the axiom of choice) that the chromatic number $\chi(\R^2)$ of the plane is equal to the supremum of the chromatic number of finite subgraphs. On the other hand, the analogous result is not known for the fractional chromatic number $\chi_f(\R^2)$. Therefore, it makes sense to introduce the following two different notations. 

\medskip

Let $\chi_f(\R^2):=\inf_{\gamma\in fc(\R^2)}\{\sum \gamma(I)\}$ denote
the {\it fractional chromatic number of the plane}, where $fc(\R^2)$ is the collection of all possible fractional colourings of the plane, i.e. nonnegative weight functions $\gamma$, such that $\gamma(I)>0$ occurs  for finitely many independent sets $I$ only, and $\sum_{x\in I}\gamma(I)\ge 1$ for each $x\in \R^2$. 
A standard approximation argument shows that it makes no difference if we allow $\gamma(I)$ to take positive values for infinitely many sets $I$. Also, in certain textbooks (e.g. \cite{sche}) $\chi_f(\R^2)$ is defined as the infimum of fractions $\frac{a}{b}$ such that a $b$-fold colouring of the plane with $a$ colours exists. It is easy to see that this definition is equivalent to the one above. 

\medskip

We define the {\it finitary fractional chromatic number of the plane} as $\chi_{f,0}(\R^2):=\sup\{ \chi_f(G): \ G\subseteq \R^2, \ G \ \text{is finite}\}$. It is clear that $\chi_{f,0}(\R^2)\le \chi_f(\R^2)$, but it is not known whether equality holds. The main result of this paper is the lower bound $4\le \chi_{f,0}(\R^2)$.  As mentioned in the Introduction, the previously known best (unpublished) bound was $3.98\le \chi_{f,0}(\R^2)$ (see \cite{Pa22}). With our new lower bound we conclude the following chain of inequalities (for the upper bound see \cite{hod, sche}) : 

\begin{equation}\label{chain1}
4\le \chi_{f,0}(\R^2)\le \chi_f(\R^2)\le 1/\delta_\mathrm{Croft}=4.35987...
\end{equation} 

\medskip

Similar notions can be introduced for the geometric chromatic number. First, the {\it finitary geometric chromatic number of the plane} is defined as $\chi_{gf,0}(\R^2)=\sup\{ \chi_{gf}(G): \ G\subseteq \R^2, \ G \ \text{is finite}\}$. Second, we could introduce $\chi_{gf}(\R^2)$ as the infimum of possible weights of geometric fractional colourings of the Euclidean plane, but this notion will not be needed in the present paper \footnote{In fact, it is not very difficult to show that geometric fractional colourings of $\R^2$  do not exist with finite weight. As such, the notion of $\chi_{gf}(\R^2)$ is meaningless. 
}. 

\medskip

Finally, we introduce the {\it finitary independence ratio of the plane} as 
$\alpha_{1}(\R^2)=\inf \{ \alpha_1(G): \ G\subseteq \R^2, \ G \ \text{is finite}\}$. Correspondingly, we define the {\it finitary Hall ratio} of the plane as $\rho(\R^2)=\frac{1}{\alpha_1(\R^2)}$. For comparison, the {\it measurable independence ratio} $m_1(X)$ of any subset $X\subseteq \R^2$ of finite positive Lebesgue measure is defined as $\sup_A \lambda(A)/\lambda(X)$, where $\lambda$ denotes the Lebesgue measure, and the supremum is taken over all measurable independent (also known as 1-avoiding) subsets $A\subseteq X$. The measurable independence ratio of $\R^2$ is then defined as $m_1(\R^2)=\inf_X m_1(X)$, where the infimum is taken over all sets $X$ of finite positive measure. An equivalent definition of $m_1(\R^2)$ is given in \cite{ambrus2023density} as the maximal upper density of a measurable 1-avoiding subset of $\R^2$. 

\medskip

Another main result of this paper is Theorem 2, showing the equality $\rho(\R^2)=\chi_{f,0}(\R^2)$. Using our new bound, $4\le \chi_{f,0}(\R^2)$, this implies $\alpha_1(\R^2)\le 1/4$. As such, we conclude the following chain of inequalities:

\begin{equation}\label{chain}
0.22936...=\delta_\mathrm{Croft}\le m_1(\R^2)\le  \alpha_1(\R^2)\le 1/4.   
\end{equation}

\medskip

By the results of \cite{ambrus2023density} we also know that $m_1(\R^2)\le 0.247$. Altogether, in the chain of inequalities \eqref{chain} we conjecture that the following equalities hold: 
$\delta_\mathrm{Croft}= m_1(\R^2)$ and $  \alpha_1(\R^2)= 1/4$.

\section{Fractional chromatic number and  Hall ratio}\label{sec3}

In this section we prove that $\chi_{gf,0}(\R^2)=\chi_{f,0}(\R^2)=\rho(\R^2)$. The first equality depends on a "blow-up" construction: given a finite graph $G$ with high value of $\chi_{gf}(G)$, we consider the union $G'$ of many congruent copies of $G$, and show that $\chi_f(G') $ must necessarily be almost as high as $\chi_{gf}(G)$. Interestingly, the argument uses the amenability of the group of Euclidean transformations in dimension 2, and cannot automatically be extended to higher dimensions. We do not know whether $\chi_{gf,0}(\R^d)=\chi_{f,0}(\R^d)$ holds for any $d\ge 3$. 

\medskip

The second equality, $ \chi_{f,0}(\R^2)=1/\alpha_1(\R^2)$, is proven by a fairly standard argument, using large "discrete cubes". 

\begin{theorem}\label{fcng}
$\chi_{gf,0}(\R^2)=\chi_{f,0}(\R^2)$.
\end{theorem} 

\begin{proof}
It is clear that $\chi_{f,0}(\R^2)\le \chi_{gf,0}(\R^2)$, and we only need to prove the reverse inequality. To this end, for any finite unit distance graph $H$, and any $\eps> 0$, we will construct another (typically much larger) finite graph $G$ such that $\chi_f(G)\ge \chi_{gf}(H)-\eps$.

\medskip

We will call any subset $S\subseteq H$ a {\it geometric shape in H}, if $S$ contains at least two points. 
Let $\varphi$ be a Euclidean transformation in the plane (i.e. a distance preserving transformation) for which there exist some geometric shapes $S, S'\subseteq H$ such that $\varphi(S)=S'$ (note that $S=S'$ is not excluded). Let $T$ denote the set of all such transformations. Clearly, $T$ is a finite set, contains the identity operator $I$, and is symmetrical with respect to $I$, i.e.~for every $\varphi \in T$, we have $\varphi^{-1}\in T$.   

\medskip

Consider the group $K$ of Euclidean transformations generated by $T$. Clearly, $K$ is countable and solvable. It is known (see e.g. \cite[Corollary 2.4 and Theorem 3.6]{gar}, or \cite[Proposition 7 and Theorem 1]{tao}) that all solvable groups are amenable, and all amenable groups have the F{\o}lner property. As such, there exists a sequence of sets $R_k\subseteq K$ such that 
\begin{equation}\label{f1}
\frac{|R_kT\triangle R_k|}{|R_k|}\to 0, 
\end{equation}

where $A\triangle B$ denotes symmetric difference. 

\medskip

We remark here that
the F{\o}lner property is no longer true for Euclidean transformations in dimensions $d\ge 3$, so the proof breaks down in those dimensions (but the statement of the theorem may be still be true). Also, given the set $T$, it is not hard to construct such a sequence $R_k$, if we wish to make the proof constructive. Finally, note that due to the fact $I\in T$, we have $R_k\subseteq R_kT$ for every $k\ge 1$, which simplifies \eqref{f1} to 
\begin{equation}\label{tk}
\frac{|R_kT\setminus R_k|}{|R_k|}\to 0. 
\end{equation}

\medskip

Let $G_k=\cup_{\sigma\in R_k}\sigma H$. We will prove that for large enough $k$, we have $\chi_f(G_k)\ge \chi_{gf}(H)-\eps$. 

\medskip

Let $\gamma$ be a fractional colouring of $G_k$ with minimal weight, i.e.~$\sum_S \gamma(S)=\chi_f(G_k)$. By \eqref{tfc} we can assume that $\gamma$ is regular, i.e.~it satisfies \eqref{agix_prime}. Let $\overline{\gamma}$ be the aggregate functional corresponding to $\gamma$. For every $\sigma\in R_k$, the graph $\sigma H$ is a subgraph of $G_k$, and therefore the restriction  $\overline{\gamma}|_{\sigma H}$ is an aggregate functional on $\sigma H$. This aggregate functional can be naturally pulled back to $H$ as follows: for every $S\subseteq H$ let $\sigma^{-1}(\overline{\gamma}|_{\sigma H})(S)=\overline{\gamma}(\sigma S)$. Note that by equation \eqref{weight}, all the aggregate functionals $\sigma^{-1}(\overline{\gamma}|_{\sigma H})$ have the same weight on $H$ as $\overline{\gamma}$ had on $G_k$, namely $\chi_f (G_k)$.  

\medskip

Finally, the average
\begin{equation}\label{aver}
\overline{\gamma_0}=\frac{1}{|R_k|}\sum_{\sigma\in R_k} \sigma^{-1}(\overline{\gamma}|_{\sigma H})
\end{equation} 
is also an aggregate functional on $H$, because fractional colourings and the corresponding aggregate functionals form a convex set. Note that $\overline{\gamma_0}$ still has weight $\chi_f(G_k)$ on $H$, and it automatically satisfies \eqref{noneg} and \eqref{agix_prime}. It remains to show that it {\it nearly} satisfies \eqref{congr}, i.e.~it satisfies \eqref{delta} with some $\delta_k$ converging to 0 as $k\to \infty$. 

\medskip

Let $Y, Y'\subseteq H$ be two congruent geometric shapes in $H$. Then there exists a transformation $\tau\in T$ such that $\tau(Y)=Y'$. Let us compare the values of $\overline{\gamma_0}$ on $Y$ and $Y'$. By definition, we have 
\begin{equation}\label{y}
\overline{\gamma_0}(Y)=
\frac{1}{|R_k|}\sum_{\sigma\in R_k} \overline{\gamma}(\sigma Y)
\end{equation}
and 
\begin{equation}\label{yp}
\overline{\gamma_0}(Y')=
\frac{1}{|R_k|}\sum_{\sigma\in R_k} \overline{\gamma}(\sigma \tau Y) .  
\end{equation}

Many terms in the two sums are identical, the exceptions being when $\sigma\in R_k$ but $\sigma \tau \notin R_k$, and when $\sigma\in R_k$ but $\sigma \tau^{-1}\notin R_k$. All of these exceptions are contained in $R_kT\setminus R_k$, due to the fact that $\tau, \tau^{-1}\in T$. Moreover,  all the terms in the sums above are less than 4.36 due to \eqref{fcnu}. Therefore, 
\begin{equation}\label{est}
|\overline{\gamma_0}(Y)-\overline{\gamma_0}(Y')|<\frac{4.36 |R_kT\setminus R_k|}{|R_k|}=\delta_k
\end{equation}

The proof is completed by observing that $\delta_k\to 0$ by \eqref{tk}, and therefore $\chi_f(G_k)> \chi_{gf}(H)-\eps_k$ with some $\eps_k\to 0$, due to \eqref{delta}.
\end{proof}

\medskip

In the next step, we establish the equality of the finitary fractional chromatic number and the finitary Hall ratio of $\R^2$. 

\begin{theorem}\label{alpha}
$\chi_{f,0}(\R^2)=\rho(\R^2).$
\end{theorem}
\begin{proof} 
Recall that the finitary Hall ratio of the plane was defined as $\rho(\R^2)=1/\alpha_1(\R^2)$.  
The inequality $\chi_{f,0}(\R^2)\ge 1/\alpha_1(\R^2)$ follows immediately from \eqref{fcna}. Therefore, we only need to  prove that 
$\chi_f(G)\le 1/\alpha_1(\R^2)$ for any finite graph $G$. 

\medskip

Let $G$ be a finite graph with vertices $x_1, \dots, x_n\in \R^2$. The rational linear combinations $\sum_{i=1}^n \beta_i x_i$ form a vector space $V$ over $\Q$. Let $v_1, \dots, v_d$ be a basis of $V$. We may assume, by rescaling the vectors $v_i$ if necessary, that in the decompositions $x_j=\sum_{i=1}^d \xi_i^{(j)} v_i$, all coefficients $\xi_i^{(j)}$ are integers. Furthermore, there is some fixed $k\in \Z$ such that $|\xi_i^{(j)}|\le k$ for all $1\le i\le d$, $1\le j\le n$.

\medskip

Consider now the lattice $L=\{\sum_{i=1}^d \beta_iv_i : \beta_i\in \Z\}$, and let $A=\{\sum_{i=1}^d \beta_iv_i \ : \  -N\le \beta_i\le N$\} be a large ``cube'' of size $|A|=(2N+1)^d$ in this lattice. By definition, there exists an independent set $B\subseteq A$ of size $|B|\ge \alpha_1(\R^2) |A|$.  

\medskip

For any $t\in L$, let $M_t=G\cap (B+t)$. Observe that $M_t$ is an independent subset of $G$, and it is non-empty if and  only if $t\in G-B.$ In this case, the  
coordinates of $t$ have absolute value $\le N+k$, so the number of such vectors $t$ is not more than $(2N+2k+1)^d$. Let us define the averaged  counting function $\gamma$ as follows: for any $S\subseteq G$ let 
\begin{equation}\label{count}
\gamma(S)=\frac{1}{|B|}|\{t: M_t=S\}|    
\end{equation}

\medskip

We claim that $\gamma$ is a fractional colouring of $G$.  Indeed, for any $x\in G$ we have
$$\sum_{x\in S} \gamma(S)=
\frac{1}{|B|} \sum_{x\in S}
|\{t: M_t=S\}|=\frac{1}{|B|}
|\{t : t\in x-B\}|=
1.$$

\medskip

The weight of the fractional colouring $\gamma$ is given by 
$\sum_S \gamma(S)=\frac{|G-B|}{|B|}\le \frac{(2N+2k+1)^d}{\alpha_1(\R^2) (2N+1)^d}\to 1/\alpha_1(\R^2)$ as $N\to \infty$. \end{proof} 

\medskip

The significance of Theorem \ref{alpha} is twofold.

\medskip

First, combined with Corollary \ref{cor} below, it shows that for every $\eps>0$  there exist a finite graph $G\subseteq \R^2$ such that the largest independent subset of $G$ is only of size $(\frac{1}{4}+\eps)|G|$. (The proof of Theorem \ref{alpha} implies that such a graph $G$ can be taken to be a large "discrete cube" in the lattice generated by a finite graph  $H$ with $\chi_f(H)\approx 4.$) This immediately implies that the upper density of any measurable 1-avoiding subset of the plane is $\le 1/4$. However, in order to prove the strict inequality $m_1(\R^2)<1/4$, Conjecture \ref{conj} below suggests that the Fourier analytic tools in \cite{ambrus2023density} were necessary. 

\medskip

Second, it was proved in \cite{ambrus2023density} that $m_1(\R^2)<1/4$, therefore, if Conjecture \ref{conj} turns out to be true, we obtain the following very interesting property of the unit distance graph of the plane: $\alpha_1(\R^2)\ne m_1(\R^2)$, i.e.~the finitary independence ratio and the measurable independence ratio of $\R^2$ do not agree. This would be an interesting example of a natural infinite graph where a measurable and a non-measurable parameter of the graph are different. 

\section{Constructing a graph $G$ with $\chi_{gf}(G)= 4$}\label{sec4}

In this section we present a finite planar  unit distance graph $G_{27}$ on 27 vertices such that $\chi_{gf}(G_{27})=4$. The graph was found by a computer search. We describe the details of the search in \appendixmacro{Appendix~A}. 

\medskip

The vertices of the graph $G_{27}$ all belong to the {\it Moser lattice}, $L_\mathrm{Moser}$. In order to define $L_\mathrm{Moser}$ we will now identify $\R^2$ with the set of complex numbers $\C$.

\begin{defi}
    Let $\omega_1=\frac12 + i\frac{\sqrt3}{2}$ and $\omega_3=\frac56 + i\frac{\sqrt{11}}{6}$. The \emph{Moser lattice} is the additive subgroup defined by
    $$
    L_\mathrm{Moser} = \{a + b\omega_1 + c\omega_3 + d\omega_1\omega_3:a,b,c,d\in\mathbb Z\}\subseteq\mathbb C.
    $$
\end{defi}

\begin{remark}\label{rem:Moser is lattice}

The field extension $\mathbb Q[\omega_1,\omega_3]/\mathbb Q$ has degree 4, and the 4-element set $\{1, \omega_1, \omega_3,\allowbreak \omega_1\omega_3\}$ is 
independent over $\mathbb Q$. Therefore, any quadruple of integers $(a, b, c, d)$ uniquely determines an element of $L_\mathrm{Moser}$. 

\end{remark}

\medskip

\begin{defi}\label{g27}
Each column in the following table gives the Moser coefficients $(a, b, c, d)$ of a vertex of the graph $G_{27}$:
\setcounter{MaxMatrixCols}{30}
$$
\begin{pmatrix}
1 & 0 & 2 & 2 & 1 & 2 & 1 & 1 & 1 & 0 & 3 & 3 & 1 & 2 & 2 & 1 & 0 & 0 & 0 & 3 & 2 & 3 & 1 & 2 & 1 & 2 & 3\\
4 & 4 & 3 & 3 & 3 & 3 & 4 & 2 & 3 & 4 & 3 & 2 & 3 & 3 & 2 & 3 & 2 & 3 & 2 & 0 & 1 & 1 & 1 & 1 & 2 & 2 & 1\\
2 & 3 & 0 & 1 & 2 & 2 & 2 & 3 & 3 & 3 & 0 & 1 & 1 & 1 & 2 & 2 & 3 & 3 & 4 & 1 & 1 & 1 & 2 & 2 & 2 & 2 & 0\\
0 & 0 & 1 & 1 & 1 & 1 & 1 & 1 & 1 & 1 & 2 & 2 & 2 & 2 & 2 & 2 & 2 & 2 & 2 & 3 & 3 & 3 & 3 & 3 & 3 & 3 & 4
\end{pmatrix}
$$
\end{defi}

\medskip

\noindent The graph $G_{27}$ is depicted in Figure \ref{fig_g27}. Unfortunately, it does not give much geometric intuition to the following result. 

\begin{figure}
  \centering
  \includegraphics[width=0.75\linewidth]{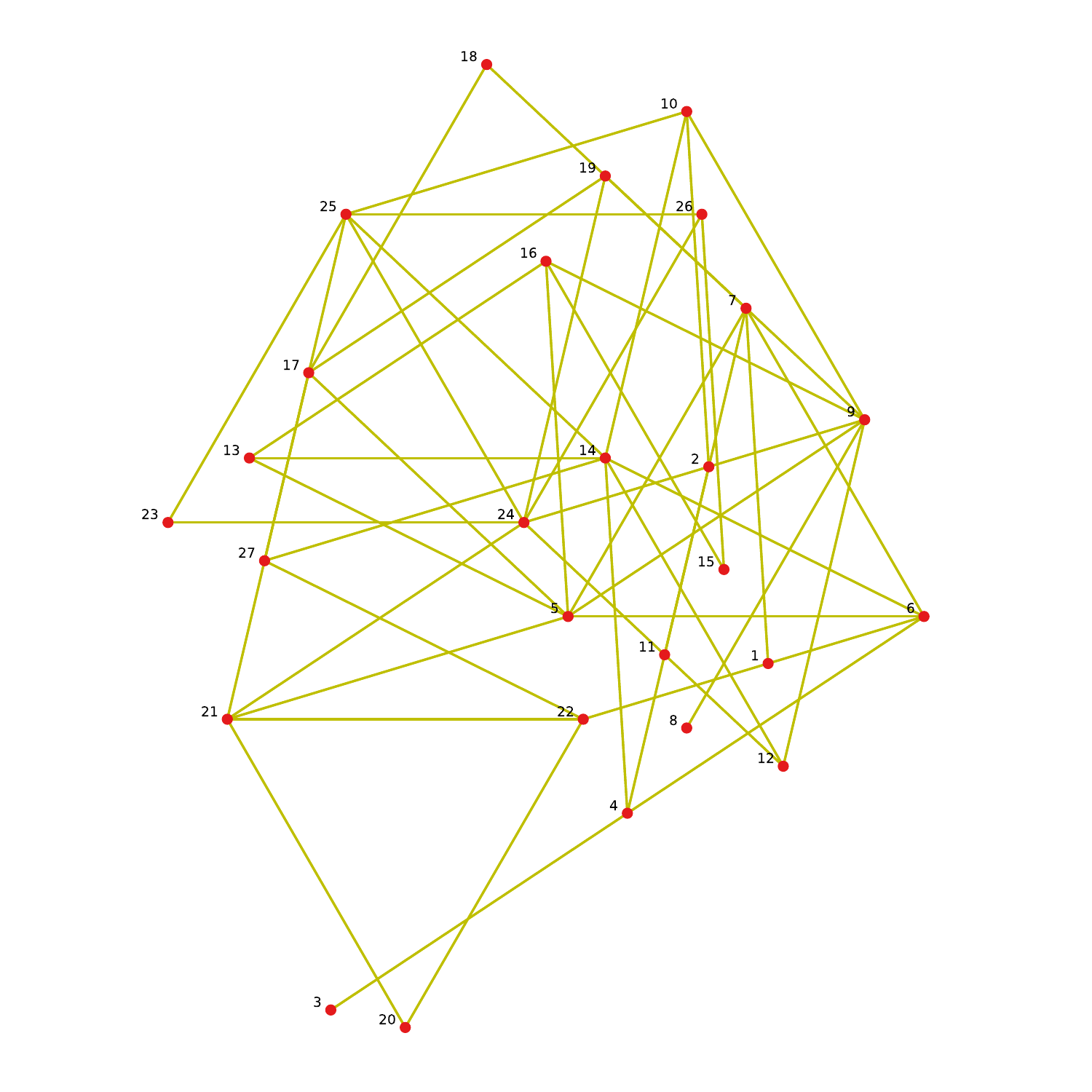}
\vspace{-10pt}
\caption{The unit distance graph $G_{27}$.}
\alt{The unit distance graph $G_{27}$.}
\label{fig_g27}
\end{figure}

\medskip

\begin{theorem}\label{gfcn4}
The geometric fractional chromatic number of $G_{27}$ is
$\chi_{gf}(G_{27}) = 4$.
\end{theorem}

\begin{proof} 

In order to determine $\chi_{gf}(G_{27})$ we first need to find all independent subsets of $G_{27}$, and all geometric congruences $Y\cong Y'$ among independent subsets $Y, Y'$. 

\medskip

In supplementary material \cite{web}, we provide the enumeration of all independent subsets of $G_{27}$, and all  congruences among them, as well as the linear programs determining $\chi_{gf}(G_{27})$. 
In order to make all the verification material in \cite{web}  transparent, we describe here the specific form and notation of the linear programs that we are using to determine $\chi_{gf}(G_{27})$. 

\medskip

The primal LP problem is defined by  
constraints \eqref{noneg}, \eqref{agix_prime} and \eqref{congr}. The variables are given by a nonnegative vector $x$ of dimension 182304. We encode \eqref{agix_prime} and \eqref{congr} in a vector $e$ and a matrix $C$, as follows. The vector $e$ is a 0-1 vector of length 182304, with ones corresponding to independent sets containing vertex number 1. By equation \eqref{agix_prime}, we have $\langle e, x\rangle=1$. For other vertices, equation \eqref{agix_prime} will be contained in the matrix $C$, which encodes the geometric congruences $Y\cong Y'$. Given two congruent subsets $Y$ and $Y'$, the row $r$ of $C$ encoding this congruency is a vector of length 182304 with +1 entries at coordinates corresponding to independent sets containing $Y$ and not containing $Y'$, and -1 entries at those containing $Y'$ and not containing $Y$. Constraint \eqref{congr} is then given by $\langle r, x\rangle =0$.   Note here that when $|Y|=|Y'|=1$, this constraint ensures that the total weight falling on each vertex is the same, and hence it is enough to postulate \eqref{agix_prime} for vertex 1, as above. As such, $\chi_{gf}(G_{27})$ is determined by minimizing $\langle {\bf{1}}, x\rangle $ subject to 
$x\ge 0$, $\langle e, x\rangle=1$ and $C.x=0$ (the boldface $\bf{1}$ denotes the all-one vector).

\medskip

This linear program has $182304$ variables (the number of independent sets), and $1+16855$ linear constraints (the 1 corresponding to the vector $e$, and the 16855 corresponding to geometric congruences $Y\cong Y'$;  in fact, we only list a minimal "spanning" set of congruences that imply all existing congruences by transitivity, reflectivity, and symmetry). 

\medskip

We first observe that $\chi_{gf}(G_{27}) \leq 4$, because $G_{27}$ has a (non-fractional) 4-colouring that is also an element of $gfc(G_{27})$. 
Assuming the ordering of vertices as in Definition \ref{g27}, the following vector assigns a colour class to each vertex of $G_{27}$:
$$
\begin{pmatrix}
3 & 4 & 4 & 2 & 1 & 4 & 2 & 4 & 3 & 1 & 4 & 1 & 2 & 3 & 2 & 4 & 4 & 3 & 2 & 4 & 2 & 3 & 1 & 4 & 2 & 3 & 4
\end{pmatrix}
$$

Having access to all congruences $Y\cong Y'$, it can be verified (as in \cite{web}) that this proper 4-colouring is indeed an element of $gfc(G_{27})$, and hence,  $\chi_{gf}(G_{27}) \leq 4$.

\medskip

By linear programming duality, lower bounding $\chi_{gf}(G_{27})$ can be achieved by finding a solution to the dual linear program. In particular, to prove that the optimal value of the primal LP problem is at least 4, we need to construct a vector $z$ of length 1+16855 such that the first coordinate of $z$ is exactly 4, and the other coordinates are arbitrary, in such a way that the inequality
$z^T (e,C)\le \bf{1}$  holds.

Now, introduce the notation $-y$  for the vector containing the last 16855 coordinates of $z$, i.e. $z=(4, -y)$. Then the equation
$z^T (e, C)\le \bf{1}$ reads as $4e - y^T C \le \bf{1}$, or $y^T C - 4 e +\bf{1} \ge \bf{0}$. 
In \cite{web} we present such a (rational) witness $y$, and a computer code verifying that $y$ indeed is a solution to this system of inequalities. The witness $y$ has $16855$ coordinates, each of which is a rational number with an approximately $250$-digit numerator and denominator.  Admittedly, the existence of this particular $y$ does not provide much insight into \emph{why} $\chi_{gf}(G_{27})$ is exactly $4$. We have not yet been successful in obtaining a witness that provides more insight. In \appendixmacro{Appendix~C}, we describe how the witness $y$ was found. (We remark here that Fernando Mario de Oliveira Filho has independently found a rational solution $y_2$ to our dual linear program, however, it also contains large denominators, and hence is also unrevealing.) 

\medskip

We encourage the interested reader to explore the extensive material provided at \cite{web}, which offers all results and verification algorithms in a transparent and reproducible manner.
\end{proof}

\medskip

As a direct corollary of Theorem~\ref{gfcn4}, we obtain the statement in the title of this note. 

\begin{corollary}\label{cor}
$\chi_f(\R^2)\ge 4.$    
\end{corollary}
\begin{proof}
By Theorem \ref{fcng} and Theorem \ref{gfcn4} we have 
$\chi_f(\R^2)\ge \chi_{f,0}(\R^2)=\chi_{gf,0}(\R^2)\ge \chi_{gf}(G_{27})=4.$
\end{proof}

Finally, we state our conjecture mentioned in the introduction. 

\begin{conj}\label{conj}
$\chi_{f,0}(\R^2)=4$, and for all finite unit distance graphs $G\subseteq \R^2$ we have $\chi_f(G)<4.$     
\end{conj}

\section{Acknowledgements}

We are indebted to Gergely Ambrus and Adrián Csiszárik for valuable discussions, to Aliaksei Vasileuski
for pointing out an error in the first version of this paper, to Fernando Mario de Oliveira Filho for independently verifying that $\chi_{gf}(G_{27})=4$, and to Pjotr Buys and Raimundo Saona for valuable suggestions to improve the presentation of the paper.

M.~M.~was supported by grants NKFIH-132097 and NKFIH-146387. I.~Z.~R.~was supported by grants NKFIH-146387 and KKP-133819.
D.~V.~and P.~Zs.~were supported by the Ministry of Innovation and Technology NRDI Office within the framework of the Artificial Intelligence National Laboratory (RRF-2.3.1-21-2022-00004).

\vspace{0.3 cm}
\noindent
{\sc Máté Matolcsi} ({\em corresponding author})

\noindent
{\em HUN-REN Alfréd Rényi Institute of Mathematics, Reáltanoda u. 13-15, 1053, Budapest, Hungary,\\
and Department of Analysis and Operations Research,
Institute of Mathematics,
Budapest University of Technology and Economics,
Mûegyetem rkp. 3., H-1111 Budapest, Hungary.}

\noindent
e-mail address: \texttt{matolcsi.mate@renyi.hu}

\medskip
\noindent
{\sc Imre Z. Ruzsa}

\noindent
{\em HUN-REN Alfréd Rényi Institute of Mathematics, Reáltanoda u. 13-15, 1053, Budapest, Hungary}

\noindent
e-mail address: \texttt{ruzsa@renyi.hu}

\medskip
\noindent
{\sc Dániel Varga}

\noindent
{\em HUN-REN Alfréd Rényi Institute of Mathematics, Reáltanoda u. 13-15, 1053,  Budapest, Hungary}

\noindent
e-mail address: \texttt{daniel@renyi.hu}

\medskip
\noindent
{\sc Pál Zsámboki}

\noindent
{\em HUN-REN Alfréd Rényi Institute of Mathematics, Reáltanoda u. 13-15, 1053, Budapest, Hungary}

\noindent
e-mail address: \texttt{zsamboki.pal@renyi.hu}

\newpage

\appendix

\section{Search algorithms for finding the graph $G_{27}$}\label{app_a}

In this section we describe the computer search method we used to find a graph $X\subset\mathbb C$ with $\chi_{gf}(X)=4$. Our search algorithm is restricted to the Moser lattice $L_\mathrm{Moser}$.

\begin{nota}
    Let $\mathscr X_\mathrm{Moser}$ denote the collection of finite graphs $X\subset L_\mathrm{Moser}$.
\end{nota}

\begin{example}\label{example:moser spindle}
    One possible embedding of the Moser spindle into $L_\mathrm{Moser}$ is
    $$
    M_7=\{0, 1, \omega_1, \omega_3, \omega_1\omega_3, 1 + \omega_1, \omega_3 + \omega_1\omega_3\}\in \mathscr X_\mathrm{Moser}.
    $$
\end{example}

To make the search effective, we seek to avoid calculating $\chi_{gf}$ for geometrically congruent graphs. As a heuristic step towards that end, for $X\in\mathscr X_\mathrm{Moser}$, we transform it to a canonical form with respect to the following operations:
\begin{enumerate}
    \item Translation by an element of the Moser lattice $L_\mathrm{Moser}$.
    \item Rotation by $\frac{\pi}{3}$, that is multiplication by $\omega_1$.
    \item Reflection with respect to the line through the origin and $\frac{\omega_1 + \omega_3}{2}$.
\end{enumerate}

\subsection{Canonization of graphs in the Moser lattice}\label{s:canonization}

In this Subsection, we describe the canonization construction we use to deduplicate the collection of graphs we consider in the computer search
with respect to translation, rotation by $\frac{\pi}{3}$, and reflection.

\begin{defi}
    Let $x=a+b\omega_1+c\omega_3+d\omega_1\omega_3\in L_\mathrm{Moser}$ be a point in the Moser lattice.Then its \emph{Moser coefficients} are the quadruple $(a,b,c,d)\in\mathbb Z^4$.
\end{defi}

\begin{remark}
    By Remark \ref{rem:Moser is lattice}, there exists an isomorphism of Abelian groups
    $$
    \mathbb Z^4\xrightarrow{\phi((a,b,c,d))=a+b\omega_1+c\omega_3+d\omega_1\omega_3}L_\mathrm{Moser}.
    $$
    That is, every element $x\in L_\mathrm{Moser}$ of the Moser lattice has a unique set of Moser coefficients $\phi^{-1}(x)\in\mathbb Z^4$.
\end{remark}

\begin{example}
    Consider the embedding $M_7\subset L_\mathrm{Moser}$ of the Moser spindle shown in Example \ref{example:moser spindle}. Then the corresponding set of Moser coefficients is the following:
    $$
    \phi^{-1}(M_7)=\left\{
    \begin{pmatrix} 0 \\ 0 \\ 0 \\ 0 \end{pmatrix},
    \begin{pmatrix} 1 \\ 0 \\ 0 \\ 0 \end{pmatrix},
    \begin{pmatrix} 0 \\ 1 \\ 0 \\ 0 \end{pmatrix},
    \begin{pmatrix} 0 \\ 0 \\ 1 \\ 0 \end{pmatrix},
    \begin{pmatrix} 0 \\ 0 \\ 0 \\ 1 \end{pmatrix},
    \begin{pmatrix} 1 \\ 1 \\ 0 \\ 0 \end{pmatrix},
    \begin{pmatrix} 0 \\ 0 \\ 1 \\ 1 \end{pmatrix}
    \right\}
    $$
\end{example}

\begin{nota}\label{nota:Rot, Refl}
    Let $\mathrm{Rot}$ resp.~$\mathrm{Refl}$ denote the automorphisms of $\mathbb Z^4$ induced by the automorphisms of the Moser lattice given by rotation by $\frac{\pi}{3}$ resp.~reflection with respect to the line through the origin and point $\frac{\omega_1+\omega_3}{2}$. They can be given as matrices to multiply vectors with from the left as follows:
    $$
    \mathrm{Rot}=\begin{pmatrix}
        0 & -1 & 0 & 0 \\
        1 & 1 & 0 & 0 \\
        0 & 0 & 0 & -1 \\
        0 & 0 & 1 & 1
    \end{pmatrix},\quad
    \mathrm{Refl}=\begin{pmatrix}
        0 & 0 & 0 & 1 \\
        0 & 0 & 1 & 0 \\
        0 & 1 & 0 & 0 \\
        1 & 0 & 0 & 0
    \end{pmatrix}.
    $$
\end{nota}

\begin{construction}
    Let $X\in\mathscr X_\mathrm{Moser}$ be a finite graph in the Moser lattice. Then we obtain its \emph{canonization} $\mathrm{Canonize}(X)$ as follows:

    I. Consider the following collection of transformations of $X$:
    \begin{align*}
    \mathscr Y=
    \{&\phi^{-1}(X), \mathop{\mathrm{Rot}}\phi^{-1}(X), \dotsc, {\mathop{\mathrm{Rot}}}^5 \phi^{-1}(X),\\
    &\mathop{\mathrm{Refl}} \phi^{-1}(X), \mathop{\mathrm{Rot}}\mathop{\mathrm{Refl}} \phi^{-1}(X), \dotsc, {\mathop{\mathrm{Rot}}}^5 \mathop{\mathrm{Refl}} \phi^{-1}(X)\}.
    \end{align*}
    Out of this collection, let $Y=\mathop{\mathrm{argmin}}_{Y\in\mathscr Y}\min_{x\in Y}x$, where we order the elements of the lattice $\mathbb Z^4$ with the lexicographical ordering.

    II. Let $Y=\{(a_{i1},a_{i2},a_{i3},a_{i4}):1\le i\le n\}$. Then for $1\le j\le 4$, let $b_j=\min_{1\le i\le n}a_{ij}$. With this, we can let
    $$
    \mathrm{Canonize}(X)=\phi(\{(a_{i1}-b_1,a_{i2}-b_2,a_{i3}-b_3,a_{i4}-b_4):1\le i\le n\}).
    $$
\end{construction}

\subsection{Backtracking beam search}\label{ss:backtracking beam search}

We perform a computer search to find a graph $G_{27}\in\mathscr X_\mathrm{Moser}$ with $\chi_{gf}(G_{27})=4$. We now describe the search method.

\begin{defi}\label{defi:children}
    Let $X\in\mathscr X_\mathrm{Moser}$ be a graph that is contained in the Moser lattice. Then the set of its \emph{children} $\mathrm{Children}(X)$ is the collection of graphs of the form $\mathrm{Canonize}(X\cup\{x\})$ such that
    \begin{enumerate}
        \item We have $x\notin X$.
        \item One of the following assertions hold:
        \begin{enumerate}
            \item There exists $x_0\in X$ such that $x-x_0\in\{1,-1,\omega_1,-\omega_1,\omega_3,-\omega_3,\omega_1\omega_3,-\omega_1\omega_3\}$.
            \item There exist $x_0,x_1\in X$ such that $\{x_0,x_1,x\}$ is a regular triangle of side length 1.
            \item There exist $x_0,x_1,x_2\in X$ such that $\{x_0,x_1,x_2,x\}$ is a parallelogram of side length 1.
        \end{enumerate}
    \end{enumerate}
    The set of \emph{parents} of $X$ is $\mathrm{Parents}(X)=\{\mathrm{Canonize}(Y):Y\subset X,\,|Y| + 1 = |X|\}$.

    For a collection $\mathscr X\subset\mathscr X_\mathrm{Moser}$ of graphs in the Moser lattice, we let $\mathrm{Children}(\mathscr X)=\cup_{X\in\mathscr X}\mathrm{Children}(X)$ and $\mathrm{Parents}(\mathscr X)=\cup_{X\in\mathscr X}\mathrm{Parents}(X)$.
\end{defi}

\begin{remark}
    One can verify by hand that if $X\in\mathscr X_\mathrm{Moser}$ is a graph in the Moser lattice, then all its children are in the Moser lattice: $\mathrm{Children}(X)\subset\mathscr X_\mathrm{Moser}$.
\end{remark}

Our search method is an improved version of the approach described in \cite[\S7]{ambrus2023density}. Starting out with the Moser spindle: $\mathscr X=\{M_7\}$, step by step we replace $\mathscr X$ by a subset of $\mathrm{Children}(\mathscr X)$ or $\mathrm{Parents}(\mathscr X)$ with high $\chi_{gf}$. We provide pseudocode for the full algorithm in Appendix B. To solve linear programs, in particular to numerically calculate geometric fractional chromatic numbers, we use the Gurobi solver \cite{gurobi}.

\begin{figure}[ht!]
  \centering
  \begin{subfigure}[t]{0.49\textwidth} 
    \begin{tikzpicture}
      \begin{axis}[
        xlabel={$n$},
        ylabel={$\chi_{gf}(G)$},
        grid=major,
        grid style={dashed,gray!30},
        mark size=2pt,
        legend style={at={(0.5,-0.2)}, anchor=north, legend columns=-1},
        width=7cm,
        height=6cm,
        ymin=3.45,
        ymax=4.05,
        ytick={3.5,3.6,3.7,3.8,3.9,4.0}
      ]

        \addplot[
          thick,
          color=black,
          mark=*,
          mark options={color=blue!65}]
        coordinates {
          (7, 3.5)
          (8, 3.583333333)
          (9, 3.75)
          (10, 3.75)
          (11, 3.749999998)
          (12, 3.833333254)
          (13, 3.85)
          (14, 3.9)
          (15, 3.928571427)
          (16, 3.935714283)
          (17, 3.950913224)
          (18, 3.961096132)
          (19, 3.966426115)
          (20, 3.968887733)
          (21, 3.972010674)
          (22, 3.976350098)
          (23, 3.98220376)
          (24, 3.987370491)
          (25, 3.992761113)
          (26, 3.996696235)
          (27, 4)
          (28, 4)
          (28, 4)
          (29, 4)

          (7, 3.5)
          (8, 3.583333333)
          (9, 3.75)
          (10, 3.75)
          (11, 3.749999998)
          (12, 3.833333254)
          (13, 3.85)
          (14, 3.9)
          (15, 3.928571427)
          (16, 3.935714283)
          (17, 3.950913224)
          (18, 3.961096132)
          (19, 3.966426115)
          (20, 3.968887733)
          (21, 3.972010674)
          (22, 3.976350098)
          (23, 3.98220376)
          (24, 3.987370491)
          (25, 3.992761113)
          (26, 3.996696235)
          (27, 4)
          (28, 4)
          (28, 4)
          (29, 4)
        };
      \end{axis}
    \end{tikzpicture}
    \caption{Full plot.}
  \end{subfigure}
  \hfill
  \begin{subfigure}[t]{0.49\textwidth} 

    \begin{tikzpicture}
      \begin{axis}[
        xlabel={$n$},
        ylabel={$\chi_{gf}(G)$},
        grid=major,
        grid style={dashed,gray!30},
        mark size=2pt,
        legend style={at={(0.5,-0.2)}, anchor=north, legend columns=-1}
        width=7cm,
        height=6cm,
        xmin=17.5,
        ymin=3.955,
        ymax=4.005,
        ytick={3.96,3.97,3.98,3.99,4.0}
    ]

        \addplot[
          thick,
          color=black,
          mark=*,
          mark options={color=blue!65}]
        coordinates {
          (17, 3.950913224)
          (18, 3.961096132)
          (19, 3.966426115)
          (20, 3.968887733)
          (21, 3.972010674)
          (22, 3.976350098)
          (23, 3.98220376)
          (24, 3.987370491)
          (25, 3.992761113)
          (26, 3.996696235)
          (27, 4)
          (28, 4)
          (28, 4)
          (29, 4)
        };

      \end{axis}
    \end{tikzpicture}
    \caption{Close-up of plot at $n \geq 18$.}
  \end{subfigure}
  \caption{Plot of the highest $\chi_{gf}(G)$ value encountered among $n$-vertex graphs $G$, for each value of $n$.}
  \label{fig:beam search bests}
\end{figure}

\subsection{Symbolic verification}

The search method presented above finds a specific graph $G_{27}\in\mathscr X_\mathrm{Moser}$, as shown in Figure~\ref{fig_g27}. See Section \ref{sec4} for the Moser coefficients of the vertices of this graph.

\section{Backtracking beam search full algorithm}\label{s:backtracking beam search full algorithm}

We provide the pseudocode for the full algorithm below. Note that Algorithm \ref{algorithm:backtracking beam search} has an infinite loop. This is because we were unsure how large the graphs we would need to check were. Moreover, calculating $\chi_{gf}$ takes significantly more time for larger graphs. Therefore, upon reaching a vertex count of 23 and a graph $X$ with $\chi_{gf}(X)\approx 3.982$, we switched to greedy search, that is beam search with $\mathtt{beam\_width}=1$, and we only looked at children. We only looked at children until we found some graph $G$ that increased $\chi_{gf}$ by at least $0.004$. With this, in 4 steps we found a graph $G_{27}$ with 27 vertices and $\chi_{gf}(G_{27})=4$. Afterwards, we made the following two sanity checks for Conjecture \ref{conj}:
\begin{enumerate}
    \item We checked all children of $G_{27}$ that can be gotten by adding a vertex that is of distance 1 to at least two vertices of $G_{27}$, the method we used in \cite[\S7]{ambrus2023density} and
    \item We checked all children of $G_{27}$ and a few hundred grandchildren with the notion of children as defined here, in Definition \ref{defi:children}.
\end{enumerate}
All of these descendants $G$ have $\chi_{gf}(G)=4$. See Figure \ref{fig:beam search bests} for the highest $\chi_{gf}(G)$ values encountered among $n$-vertex graphs, for each $7\le n\le29$.

\SetKw{Continue}{continue}
\SetKw{Return}{return}
\SetKwFor{Loop}{Loop}{}{EndLoop}
\SetKwProg{Def}{def}{:}{}

\begin{algorithm}
    \DontPrintSemicolon
    \Def{\texttt{get\_beam}($\mathscr X\subset\mathscr X_\mathrm{Moser},\,\mathtt{beam\_width}\in\mathbb Z$)}{
        \If{$|\mathscr X|\le\mathtt{beam\_width}$}{
        \Return $\mathscr X$\;
        }
        \Else{
        $\mathtt{threshold} \longleftarrow \chi_{gf}(X)$ where $X$ is the entry in $\mathscr X$ of index \texttt{beam\_width} (in 1-based indexing) when listed by descending order with respect to $\chi_{gf}$\;
        \Return $\{X\in\mathscr X:\chi_{gf}(X)\ge\mathtt{threshold}\}$\;}
    }
    $\mathtt{beam\_width} \longleftarrow 100$\;
    $\mathtt{best}(7)=\chi_{gf}(M_7)$\;
    $\mathtt{best}(i)=0\text{ for all }i>7$\;
    $\mathtt{direction \longleftarrow 1}$\;
    $\mathtt{forward}(i)=0\text{ for all }i\ge7$\;
    $\mathtt{num\_vertices\_min} \longleftarrow 7$\;
    $\mathtt{seen}(7)=\{M_7\}$\;
    $\mathtt{seen}(i)=\emptyset\text{ for all }i>7$\;
    $\mathscr X \longleftarrow \{M_7\}$\;
    $i \longleftarrow 8$\;
    \Loop{}{
        \If{$\mathtt{direction}=1$}{
            $\mathscr X \longleftarrow \mathrm{Children}(\mathscr X)\setminus\mathtt{seen}(i)$\;
        }
        \Else{
            $\mathscr X \longleftarrow \mathrm{Parents}(\mathscr X)\setminus\mathtt{seen}(i)$\;
        }
        \If{$\mathscr X=\emptyset$}{
            $\mathtt{direction} \longleftarrow 1$\;
            $i \longleftarrow \max\{i\ge7:\mathtt{seen}(i)\ne\emptyset\} + 1$\;
            $\mathscr X \longleftarrow \mathtt{get\_beam}(\mathtt{seen}(i))$\;
            \Continue\;
        }
        $\mathtt{seen}(i) \longleftarrow \mathtt{seen}(i) \cup \mathscr X$\;
        $\mathtt{new\_best} \longleftarrow \max_{X\in\mathscr X}\chi_{gf}(X)$\;
        \If{$\mathtt{new\_best}\ge\mathtt{best}(i)\ \mathbf{and}\ i>\mathtt{num\_vertices\_min}$}{
            $\mathtt{best}(i) \longleftarrow \mathtt{new\_best}$\;
            $\mathtt{direction} \longleftarrow -1$\;
            $\mathscr X \longleftarrow \mathtt{get\_beam}(\mathscr X)$\;
            $\mathtt{forward}(i) \longleftarrow \mathtt{get\_beam}(\mathtt{forward}(i) \cup \mathscr X)$\;
        }
        \Else {
            $\mathtt{direction} \longleftarrow 1$\;
            $\mathscr X \longleftarrow \mathtt{get\_beam}(\mathtt{forward}(i)\cup\mathscr X)$\;
            $\mathtt{forward}(i) \longleftarrow \emptyset$\;
        }
        $i \longleftarrow i + \mathtt{direction}$\;
        
    }
    \caption{Backtracking beam search}
    \label{algorithm:backtracking beam search}
\end{algorithm}

\vskip 3cm

\section{The dual witness}\label{dw}

As described in Section~\ref{sec4}, to prove Theorem~\ref{gfcn4}, we need to present any point $y$ of the dual feasible region $y^T C - 4 e +\bf{1} \ge \bf{0}$. Let us denote this polytope by $W$.

Employing an LP solver that operates on rational numbers is infeasible given the size of the LP we need to handle. Our numerical LP solver provides a floating point solution $\hat{y}$ to the dual  of the linear program determining $\chi_{gf}(G_{27})$. Due to unavoidable numerical errors, the existence of $\hat{y}$ only implies a lower bound $\chi_{gf}(G_{27}) > 4 - 10^{-8}$.

Our strategy is as follows: we take a numerical solution $\overline{y}$ lying very close to $W$, interpret it as a rational vector, and orthogonally project it onto $W$. We start by selecting a set of inequalities that are ``supposed to be equalities'', that is, are either violated, or sharp up to a numerical threshold $\varepsilon=10^{-12}$. Then, using symbolic QR decomposition, we orthogonally project $\overline{y}$ onto the affine subspace defined by these equalities. We denote the result of the projection by $y$. For $y$, the selected inequalities are now sharp, and thanks to the fact that $||y-\overline{y}||_2$ is small, the rest of the inequalities remain respected.

The obvious candidate for $\overline{y}$ is $\hat{y}$, the inexact numerical solution of the dual LP, and indeed $\hat{y}$ lies very close to $W$. However, for $\hat{y}$ it is not obvious how to correctly choose the set of inequalities that we enforce to be equalities when doing the linear projection. For this reason, we solve one more linear program: we choose $\overline{y}$ to be the numerical solution of the dual LP which minimizes $||y||_\infty$ while $y\in W$. For this vector, the inequalities have the desirable max-margin property that 168 inequalities are sharp up to precision $10^{-12}$, and the rest all have significantly larger slack. Orthogonally projecting $\overline{y}$ onto the affine subspace defined by the 168 equations leads to a $y$ that satisfies all our inequalities, proving $\chi_{gf}(G_{27}) \geq 4$. See \cite{web} for code and data sufficient to reproduce this verification.


\begin{thebibliography}{99}


\bibitem{ambrus2023density}
G. Ambrus, A. Csisz{\'a}rik, M. Matolcsi, D. Varga, P. Zs{\'a}mboki, {\em The density of planar sets avoiding unit distances.}
Mathematical Programming, 
{\bf 207}, 303-327, (2024)

\bibitem{BePS21} T. Bellitto, A. P\^{e}cher, A. S\'{e}dillot, {\em On the density of sets of the Euclidean plane avoiding distance 1.}  Discrete Mathematics \& Theoretical Computer Science {\bf 23} (1), 5153 (2021)

\bibitem{blu}
A. Blumenthal, B. Lidick\'y, R. R. Martin, S. Norin, F. Pfender, J. Volec: {\em Counterexamples
to a conjecture of Harris on Hall ratio}, SIAM Journal on Discrete Mathematics, {\bf 36}(3):1678-1686, (2022)

\bibitem{edb}
N. G. de Bruijn, P. Erd\H os: {\em A colour problem for infinite graphs and a problem in the theory of relations}, Nederl. Akad. Wetensch. Proc. Ser. A, 54: 371-373, 1951.

\bibitem{cra}
D.W. Cranston, L. Rabern: {\em The fractional chromatic number of the plane}, Combinatorica {\bf 37}, 837-861 (2017). 

\bibitem{Cr67} H.T. Croft, {\em Incidence incidents}. Eureka {\bf 30}, 22--26 (1967)

\bibitem{dvo}
Z. Dvo\v{r}\'ak, P. O. de Mendez, H. Wu: {\em 1-subdivisions, the fractional chromatic number
and the Hall ratio}, Combinatorica, {\bf 40}:759-774, (2020)

\bibitem{gar}
A. Garrido, {\em An introduction to amenable groups}, Lecture Notes, 
\url{https://www.math.uni-duesseldorf.de/~garrido/amenable.pdf}


\bibitem{deG18} A. de Grey, {\em
The chromatic number of the plane is at least 5}.  Geombinatorics {\bf 28} (1), 18--31  (2018)

\bibitem{gurobi}
Gurobi Optimization, LLC,
{\em Gurobi Optimizer Reference Manual}.
\url{https://www.gurobi.com}
(2023)

\bibitem{har}
D. G. Harris: {\em Some results on chromatic number as a function of triangle count}, SIAM Journal on Discrete Mathematics, {\bf 33}(1):546-563, 2019.

\bibitem{hod}
R. Hochberg, P. O'Donnell: {\em A large independent set in the unit distance graph},  Geombinatorics 2(4):83-84, (1993).

\bibitem{jan}
B. Janzer, R. Steiner, B. Sudakov: {\em Chromatic number and regular subgraphs}, arXiv
preprint arXiv:2410.02437, 2024.

\bibitem{john}
P. D. Johnson Jr: {\em The fractional chromatic number, the Hall ratio, and
the lexicographic product}, Discrete Mathematics, {\bf 309}, pp. 4746-4749, (2009)

\bibitem{mah}
S. L. Mahan: {\em The fractional chromatic number of the plane}, Master’s thesis, Uni-
versity of Colorado Denver, 1995.

\bibitem{nes}
J. Ne\v{s}et\v{r}il, P. Ossona de Mendez: {\em Grad and classes with bounded
expansion I. Decompositions}, European J. Combin., 29 (2008), pp. 760–
776.

\bibitem{Pa22} J. Parts, {\em PolyMath16, Comment 37575}. \url{https://dustingmixon.wordpress.com/2021/02/01/polymath16-seventeenth-thread-declaring-victory/\#comment-37575}

\bibitem{PM22} PolyMath16: Hadwiger-Nelson problem. \url{https://asone.ai/polymath/index.php?title=Hadwiger-Nelson\_problem} (2022)

\bibitem{sche}
E. R. Scheinerman, D. H. Ullman: {\em Fractional Graph Theory: A rational approach to the theory of graphs}, Wiley-Interscience Series in Discrete Mathematics and Optimization, John Wiley \& Sons, Inc., New York, 1997.

\bibitem{steiner}
R. Steiner: {\em Fractional chromatic number vs. Hall ratio}, arXiv:2411.16465


\bibitem{tao}
Terence Tao: {\em Some notes on amenability}, 
https://terrytao.wordpress.com/2009/04/14/some-notes-on-amenability/

\bibitem{web}
Supplementary material for `The fractional chromatic number of the plane is at least 4'. \url{https://bit.ly/fcn-4} 2024.

\end{thebibliography}
\end{document}